\documentclass[review,hidelinks,onefignum,onetabnum]{siamart220329}
\usepackage[top=2in,bottom=1in, left=1in, right=1in]{geometry}
\usepackage{subfigure,multirow}
\usepackage{amsmath,amssymb}
\usepackage{graphicx}
\usepackage{float}
\usepackage{booktabs}
\usepackage{multirow}
\usepackage[justification=centering]{caption}

\usepackage{lipsum}
\usepackage{amsfonts}
\usepackage{graphicx}
\usepackage{epstopdf}
\usepackage{algorithmic}
\ifpdf
  \DeclareGraphicsExtensions{.eps,.pdf,.png,.jpg}
\else
  \DeclareGraphicsExtensions{.eps}
\fi

\newsiamremark{remark}{Remark}
\newsiamremark{hypothesis}{Hypothesis}
\newsiamremark{assumption}{Assumption}
\newsiamthm{claim}{Claim}

\headers{Homogenization analysis for FSDEs}{Zhaoyang Wang, Ping Lin}

\title{Homogenization principle and numerical analysis for fractional stochastic differential equations with different scales\thanks{Submitted to the editors \today.
}}

\author{Zhaoyang Wang\thanks{School of Mathematical Sciences, Laboratory of Mathematics and Complex Systems, MOE, Beijing Normal University, Beijing 100875, China; Research Center for Mathematics, Advanced Institute of Natural Sciences, Beijing Normal University, Zhuhai, Guangdong 519087, China.
(\email{zhaoyang584520@163.com}).}
\and Ping Lin\thanks{Corresponding author. Division of Mathematics, University of Dundee, Dundee DD1 4HN, United Kingdom 
(\email{p.lin@dundee.ac.uk}).}
}

\usepackage{amsopn}

\ifpdf
\hypersetup{
  pdftitle={Homogenization principle and numerical analysis for fractional stochastic differential equations with different scales},
  pdfauthor={Zhaoyang Wang and Ping Lin}
}
\fi

\begin{document}

\maketitle

\begin{abstract}
This work is concerned with fractional stochastic differential equations with different scales. We establish the existence and uniqueness of solutions for Caputo fractional stochastic differential systems under the non-Lipschitz condition. Based on the idea of temporal homogenization, we prove that the homogenization principle (averaging principle) holds in the sense of mean square ($L^2$ norm) convergence under a novel homogenization assumption. Furthermore, an Euler-Maruyama scheme for the non-autonomous system is constructed and its numerical error is analyzed. Finally, two numerical examples are presented to verify the theoretical results. Different from the existing literature, we demonstrate the computational advantages of the homogenized autonomous system from a numerical perspective. 
\end{abstract}



\section{Introduction}
\label{section1}
Stochastic differential equations (SDEs) are a critical mathematical tool extensively used in various scientific and engineering fields to model dynamic systems with inherent randomness or uncertainty. Traditional deterministic differential equations fall short in capturing the random perturbations presented in these systems, whereas SDEs incorporate stochastic processes to more reasonably describe the unpredictable effects observed in nature and engineering \cite{mao2007stochastic}. 

In the past two decades, fractional calculus and fractional differential equations have been extensively studied and applied in many fields \cite{podlubny1998fractional, metzler2000random, wang2024error, ainsworth2020fractional}, and stochastic systems have no exception. Sakthivel et al. \cite{sakthivel2013existence} used stochastic analysis theory to study the existence of solutions for fractional stochastic semilinear differential equations with nonlocal conditions. Li et al. \cite{li2017fractional} proposed a fractional stochastic differential equation model which is consistent with the over-damped limit of the generalized Langevin equation, and suitable for systems in contact of heat bath with subdiffusion behaviors. Furthermore, they established the existence of strong solutions and discussed the ergodicity and convergence to Gibbs measure. Recently, the well-posedness of a class of Caputo fractional stochastic differential equations is established \cite{wang2020note}. For the numerical approximation of fractional stochastic differential equations (FSDEs), the Euler-Maruyama method is simple and effective. Doan et al. \cite{doan2020euler} constructed a Euler-Maruyama type scheme for Caputo FSDEs. Yang \cite{yang2022numerical} developed a Euler-Maruyama scheme for non-autonomous systems involving Caputo-Hadamard fractional derivatives and conducted error analysis. 

On the other hand, for stochastic oscillatory systems with variables of different scales, the homogenization method \cite{weinan2011principles, blanc2023homogenization} is one of the effective tools for studying such systems. The homogenization principle (averaging principle) for stochastic differential equations was first introduced by Khasminskii \cite{Khsminskii1968}. In recent years, with the development of fractional calculus, the homogenization principle has been extended to FSDEs \cite{xu2019averaging, shen2022averaging, guo2023averaging, li2023existence}. To theoretically address the difficulties caused by the singular kernels in fractional operators, researchers have proposed various assumptions, but most of these are not applicable to long-term oscillation systems. In addition, it is worth noting that for stochastic differential equations with variables of different scales, the homogenized equation under certain conditions can indeed simplify the original equation into an autonomous system for easier analysis. However, to date, research on the homogenization of FSDEs does not appear to have shown practical benefits for solving application problems.

In this paper, we focus on the following  Caputo type fractional stochastic differential equations with different scales driven by Brownian motion:

\begin{equation}\label{the1-1}
\begin{split}
\begin{cases}
^CD_{0}^{\alpha}x_\epsilon(t)=f\left(\frac{t}{\epsilon},x_\epsilon\left(t\right)\right)dt+g\left(\frac{t}{\epsilon},x_\epsilon(t)\right)dB_t, \quad t\in[0,T]\\
x_\epsilon(0)=x_0,
\end{cases}
\end{split}
\end{equation}
where $^CD_{0}^{\alpha}$ denotes the Caputo fractional derivative with $\alpha\in(\frac{1}{2}, 1)$, $\epsilon\ll 1$ is a scale parameter. $f:\mathbb{R}^+\times\mathbb{R}^n\rightarrow\mathbb{R}^n$ and $g:\mathbb{R}^+\times\mathbb{R}^n\rightarrow\mathbb{R}^n \otimes \mathbb{R}^m$ are measurable continuous function, $B_t$ is an $m$-dimensional Brownian motion on a complete probability space $(\Omega,\mathcal{F},P)$. This model is derived from a typical dynamical system with scale separation \cite{pavliotis2008multiscale}. As previously mentioned, (\ref{the1-1}) has many theoretical results on well-posedness and homogenization principle, but there are still some important and interesting problems that deserve to be studied. Inspired by our recent work on fractional-order temporal multiscale problem \cite{wang2023fast}, we propose a weaker homogenization condition to prove the homogenization principle for FSDEs. Furthermore, we analyze and show the advantages of the homogenization method in numerical computations. Compared to existing literature, our contributions are threefold:

\textbullet \ We establish the well-posedness for the solution of (\ref{the1-1}) with the nonlinear terms unnecessarily satisfying the Lipschitz condition in the sense of $L^p$ ($p$th moment).

\textbullet \ The homogenization principle for FSDEs is proven in the sense of mean square. The proposed homogenization assumption is applicable to long-term oscillation systems.

\textbullet \ We propose an Euler-Maruyama scheme and provide error analysis, pointing out the numerical advantages of homogenization for multiscale non-autonomous stochastic systems. Numerical experiments verify the results of theoretical analysis.

The paper is organized as follows. In Section \ref{section2}, we give some definitions and assumptions. In Section \ref{section3}, we prove the well-posedness of solutions for FSDEs. In Section \ref{section4}, we establish an approximation theorem as a homogenization principle for the solutions of the concerned FSDEs. In Section \ref{section5}, we construct an Euler-Maruyama scheme and perform rigorous error analysis. Numerical examples are shown in Section \ref{section6} to verify the correctness of the theory. Some conclusions and remarks are given in final Section \ref{section7}.

\textbf{Notation.} Let $|\cdot|$ denotes the Euclidean norm in $n$-dimensional Euclidean space $\mathbb{R}^n$ and $\Vert\cdot\Vert$ be the norm of $\mathbb{R}^n \otimes \mathbb{R}^m$. For $T>0$, let $C\left([0,T];\mathbb{R}^n\right)$ be the Banach space of all $\mathbb{R}^n$-valued continuous functions on $[0,T]$, equipped with the supremum norm. We define a Banach space $L^p\left(\Omega; C\left([0,T];\mathbb{R}^n\right)\right)$ be the totality of $C\left([0,T];\mathbb{R}^n\right)$-valued random variables $x$ with the norm $\Vert x\Vert_{L^p}:=\left[\mathbb{E}\left(\sup\limits_{t\in[0,T]}|x(t)|^p\right)\right]^{1/p}$. Throughout this paper, we use $C$ to denote a generic positive constant, with or without subscript, its value may change from one line of an
estimate to the next.

\section{Preliminaries}
\label{section2}
In this section, we give the basic definitions of the fractional calculus and impose appropriate assumptions on the nonlinear terms $f$ and $g$. 

\begin{definition}
Let $f(t)$ be a differentiable function, the $\alpha\in(0,1)$ order Caputo fractional derivative is defined as \cite{Podlubny1999}
\begin{equation}\label{the2-1}
\begin{split}
D_{0^+}^{\alpha}f(t)=\frac{1}{\Gamma(1-\alpha)}\int_0^{t}\frac{f'(s)}{(t-s)^{\alpha}}ds.
\end{split}
\end{equation}

\end{definition}

\begin{definition}
The Riemann-Liouville fractional integral for a function $f(t)\in L^1[0,T]$ of order $\alpha\in(0,1)$ is denoted by \cite{Podlubny1999}
\begin{equation}\label{the2-2}
	\begin{split}
	I^{\alpha}_{0^+}f(t)=\frac{1}{\Gamma(\alpha)}\int_0^{t}(t-s)^{\alpha-1}f(s)ds.
	\end{split}
\end{equation}
\end{definition}

\begin{remark}
Based on the definition of the Caputo derivative and fractional integral, we have the following property
	\begin{equation}\label{the2-3}
		\begin{split}
			I_{0^+}^{\alpha}D_{0^+}^{\alpha}u(t)=u(t)-u(0).
		\end{split}
	\end{equation}
\end{remark}

Taking $I_{0^+}^{\alpha}$ on both sides of (\ref{the1-1}), then it becomes the following equivalent Volterra integral equation
	\begin{equation}\label{the2-4}
	\begin{split}
x_\epsilon(t)=x_0+\frac{1}{\Gamma(\alpha)}\int_0^t (t-s)^{\alpha-1}f\left(\frac{s}{\epsilon},x_\epsilon(s)\right)ds+\frac{1}{\Gamma(\alpha)}\int_0^t (t-s)^{\alpha-1}g\left(\frac{s}{\epsilon},x_\epsilon(s)\right)dB_s.
	\end{split}
\end{equation}

Next, we impose essential assumptions for nonlinear terms and introduce a lemma for the smoothness of subsequent analysis.

\begin{assumption}
\label{assumption1}
There exists a continuous and non-decreasing concave function $\psi:\mathbb{R}^{+}\rightarrow\mathbb{R}^{+}$ that satisfies $\psi(t)>0$ for $t>0$, and $\int_{0^{+}}\frac{1}{\psi(s)}ds=+\infty$, such that all $x_1, x_2\in \mathbb{R}^n$ and $t\in [0,T]$,
\begin{equation}\label{the2-5}
\begin{split}
\left|f(t,x_1)-f(t,x_2)\right|^p\vee \Vert g(t,x_1)-g(t,x_2)\Vert^p \leq \psi(\left|x_1-x_2\right|^p),
\end{split}
\end{equation}
where $p\geq 2$. In addition, for $f(t,0)$ and $g(t,0)$, there exists a positive constant $L$ such that
\begin{equation}\label{the2-6}
	\begin{split}
		\left|f(t,0)\right|^p\vee \Vert g(t,0)\Vert^p \leq L.
	\end{split}
\end{equation}
\end{assumption}

\begin{remark}
\label{remark2-5}
By using Assumption \ref{assumption1}, we can have the following estimate
\begin{equation}\label{the2-7}
\begin{split}
\left|f(t,x)\right|^p&=\left|f(t,x)-f(t,0)+f(t,0)\right|^p\\
&\leq 2^{p-1}\left|f(t,x)-f(t,0)\right|^p+2^{p-1}\left|f(t,0)\right|^p\\
&\leq 2^{p-1}\psi(|x|^p)+2^{p-1}L.
\end{split}
\end{equation}
We can obtain an estimate of $\Vert g(t,x)\Vert^p$ in a similar way.
\end{remark}

\begin{lemma}
\label{lemma2-6}
(The Bihari inequality \cite{ouaddah2021fractional}) Let $u:[0,T]\rightarrow [0,\infty]$ be a continuous function, $\psi:\mathbb{R}^{+}\rightarrow\mathbb{R}^{+}$ be a nondecreasing continuous function, and $v$ be a nonnegative integrable function on $[0,T]$. If there exist a constant $c>0$ such that the following integral inequality holds
\begin{equation}\label{the2-8}
\begin{split}
u(t)\leq c+\int_{0}^t v(s)\psi\left(u(s)\right)ds,
\end{split}
\end{equation}
then
\begin{equation}\label{the2-9}
	\begin{split}
		u(t)\leq W^{-1}\left(W(c)+\int_0^t v(s)ds\right), \ t\in[0,T].
	\end{split}
\end{equation}
Here, $W$ is defined by
\begin{equation}\label{the2-10}
\begin{split}
W(r)=\int_1^r \frac{1}{\psi(s)}ds,
\end{split}
\end{equation}
$W^{-1}$ is the inverse function of $W$, and for every $t\in[0,T]$,
\begin{equation}\label{the2-11}
\begin{split}
\psi(c)+\int_0^t v(s)ds \in Dom(W^{-1}),
\end{split}
\end{equation}
$Dom(W^{-1})$ is the domain of $W^{-1}$.
\end{lemma}

\section{Well-posedness}
\label{section3}
In this section, we will establish the existence and uniqueness for the solution of (\ref{the1-1}) with assumption \ref{assumption1} by using the Picard iteration method. Let $x_\epsilon:=x$ and $\epsilon=1$ without loss of generality. We then construct a sequence that satisfies
\begin{equation}\label{the3-1}
	\begin{split}
x_{k+1}(t)=x_0+\frac{1}{\Gamma(\alpha)}\int_0^t (t-s)^{\alpha-1}f\left(s,x_k(s)\right)ds+\frac{1}{\Gamma(\alpha)}\int_0^t (t-s)^{\alpha-1}g\left(s,x_k(s)\right)dB_s,
	\end{split}
\end{equation}
where $k\geq 0$ is an integer.

\begin{theorem}
\label{theorem3-1}
Suppose that Assumption \ref{assumption1} holds and $\alpha\in\left(1-\frac{1}{p},1\right]$, $p\geq 2$. Then the equation (\ref{the1-1}) has a unique solution $x(t)\in L^p\left(\Omega;C\left([0 \ T];\mathbb{R}^n\right)\right)$. Moreover, $x(t)$ has the following $p$th moment estimate
\begin{equation}\label{the3-2}
\begin{split}
\sup_{0\leq t\leq T}\mathbb{E}\left|x(t)\right|^p\leq C_{T31}
\end{split}
\end{equation}
with a constant $C_{T31}>0$ that depends on $p,L,\alpha,x_0,T$.

\end{theorem}

\begin{proof}
\textbf{Moment estimate.} Recalling the inequality
\begin{equation}\label{the3-3}
\begin{split}
\left|x_1+x_2+x_3\right|^p\leq 3^{p-1}\left(\left|x_1\right|^p+\left|x_2\right|^p+\left|x_3\right|^p\right),
\end{split}
\end{equation}
we have
\begin{equation}\label{the3-4}
\begin{split} 
\mathbb{E}\left|x_{k+1}(t)\right|^p&\leq 3^{p-1} \mathbb{E}|x_0|^p+ \frac{3^{p-1}}{\Gamma(\alpha)^p} \mathbb{E}\left|\int_0^t(t-s)^{\alpha-1}f(s,x_k(s))ds\right|^p+\frac{3^{p-1}}{\Gamma(\alpha)^p} \mathbb{E}\left|\int_0^t(t-s)^{\alpha-1}g(s,x_k(s))dB_s\right|^p\\
&=: 3^{p-1}\mathbb{E}|x_0|^p+\frac{3^{p-1}}{\Gamma(\alpha)^p}I_1+\frac{3^{p-1}}{\Gamma(\alpha)^p}I_2.
\end{split}
\end{equation}
	
By using the H\"{o}lder inequality and Assumption \ref{assumption1}, we obtain
\begin{equation}\label{the3-5}
\begin{split} 
I_1&\leq \left(\int_0^t(t-s)^{\frac{(\alpha-1)p}{p-1}}ds\right)^{p-1}\mathbb{E}\left(\int_0^t\left|f(s,x_k(s))\right|^pds\right)\\
&\leq 2^{p-1}\cdot T^{\alpha p-1}\cdot \int_0^t \psi\left(\mathbb{E}\left|x_k(s)\right|^p\right)ds+2^{p-1}L\cdot T^{\alpha p},
\end{split}
\end{equation}
where the concavity of $\psi$ is used.

For $I_2$, by using the Burkholder-Davis-Gundy inequality, we obtain
\begin{equation}\label{the3-6}
\begin{split} 
I_2&\leq C_p\mathbb{E}\left(\left(\int_0^t(t-s)^{2\alpha-2}\Vert g(s,x_k(s))\Vert^2 ds\right)^\frac{p}{2}\right)\\
&\leq C_pT^{\frac{p}{2}-1}\mathbb{E}\left(\int_0^t (t-s)^{(\alpha-1)p}\Vert g(s,x_k(s))\Vert^p ds\right)\\
&\leq C_p2^{p-1}T^{\frac{p}{2}-1}\int_0^t (t-s)^{(\alpha-1)p}\psi\left(\mathbb{E}\left|x_k(s)\right|^p\right) ds+C_p2^{p-1}L\cdot T^{\left(\alpha-\frac{1}{2}\right)p}.
\end{split}
\end{equation}
We then have from (\ref{the3-4})
\begin{equation}\label{the3-7}
	\begin{split} 
\mathbb{E}\left|x_{k+1}(t)\right|^p&\leq 3^{p-1}\mathbb{E}|x_0|^p+C_{p,L,\alpha,T}+C_{p,\alpha,T}\int_{0}^t \left(1+(t-s)^{(\alpha-1)p}\right)\psi\left(\mathbb{E}\left|x_k(s)\right|^p\right)  ds.
\end{split}
\end{equation}
	
Define $X_n:=\max\limits_{0\leq k\leq n+1}\mathbb{E}|x_k|^p\leq \max\limits_{0\leq k\leq n}\mathbb{E}|x_{k+1}|^p+\mathbb{E}|x_0|^p$. Hence, we have
\begin{equation}\label{the3-8}
	\begin{split} 
\mathbb{E}\left|X_n \right|&\leq (3^{p-1}+1)\mathbb{E}|x_0|^p+C_{p,L,\alpha,T}+C_{p,\alpha,T}\int_{0}^t \left(1+(t-s)^{(\alpha-1)p}\right)\psi\left(\mathbb{E}\left|X_n \right|\right)  ds.
	\end{split}
\end{equation}

Since $\psi$ is a concave and non-decreasing, there exists a positive number $C$ such that $\psi(x)\leq C(1+x)$, it follows that
\begin{equation}\label{the3-9}
\begin{split} 
\mathbb{E}\left|X_n \right|\leq (3^{p-1}+1)\mathbb{E}|x_0|^p+C_{p,L,\alpha,T}+C_{p,\alpha,T}\int_{0}^t \left(1+(t-s)^{(\alpha-1)p}\right)\mathbb{E}\left|X_n \right|  ds.
\end{split}
\end{equation}
By using the Gronwall inequality, it is easy to obtain
\begin{equation}\label{the3-10}
	\begin{split} 
		\mathbb{E}\left|X_n \right|\leq C_{p,L,\alpha,x_0,T}.
	\end{split}
\end{equation}

\textbf{Cauchy sequence and existence.} Next, we prove that $x_k$ is a Cauchy sequence in $L^p\left(\Omega;C\left([0,T];R^n\right)\right)$. For $m, k\geq 1$, we have
\begin{equation}\label{the3-11}
	\begin{split} 
\mathbb{E}|x_{m+1}(t)-x_{k+1}(t)|^p&\leq \frac{2^{p-1}}{\Gamma(\alpha)^p}\mathbb{E}\left|\int_0^t (t-s)^{\alpha-1}\left(f(s,x_{m}(s))-f(s,x_{k}(s))\right) ds\right|^p\\
&+\frac{2^{p-1}}{\Gamma(\alpha)^p}\mathbb{E}\left|\int_0^t (t-s)^{\alpha-1}\left(g(s,x_{m}(s))-g(s,x_{k}(s))\right) dB_s\right|^p\\
&=:\frac{2^{p-1}}{\Gamma(\alpha)^p}J_1+\frac{2^{p-1}}{\Gamma(\alpha)^p}J_2.
	\end{split}
\end{equation}
With the help of Assumption \ref{assumption1}, we estimate 
\begin{equation}\label{the3-12}
	\begin{split} 
J_1&\leq \left(\int_0^t(t-s)^{\frac{(\alpha-1)p}{p-1}}ds\right)^{p-1}\mathbb{E}\left(\int_0^t\left|f(s,x_m(s))-f(s,x_k(s))\right|^pds\right)\\
&\leq T^{\alpha p-1}\int_{0}^t\mathbb{E}\left(\psi(|x_m(s)-x_k(s)|^p)\right) ds\\
&\leq  C_T\int_{0}^t\psi\left(\sup\limits_{0\leq r\leq s}\mathbb{E}|x_m(r)-x_k(r)|^p\right)ds
	\end{split}
\end{equation}
and 
\begin{equation}\label{the3-13}
\begin{split} 
J_2&\leq C_p\int_0^t(t-s)^{(\alpha-1)p}\psi\left(\mathbb{E}|x_m(s)-x_k(s)|^p\right)ds\\
&\leq C_p\int_0^t(t-s)^{(\alpha-1)p}\psi\left(\sup\limits_{0\leq r\leq s}\mathbb{E}|x_m(r)-x_k(r)|^p\right)ds
\end{split}
\end{equation}
where the non-decreasing of $\psi$ is used.

Let $w(s)=\sup\limits_{0\leq r\leq s}\mathbb{E}\left|x_m(r)-x_k(r)\right|^p$. By (\ref{the3-11}), (\ref{the3-12}) and (\ref{the3-13}), we have
\begin{equation}\label{the3-14}
	\begin{split} 
		\mathbb{E}|x_{m+1}(t)-x_{k+1}(t)|^p&\leq C\int_{0}^t \left(1+(t-s)^{(\alpha-1)p}\right)\psi\left(w(s)\right) ds.
	\end{split}
\end{equation}
	
For $0\leq t_0\leq t$, note that
\begin{equation}\label{the3-15}
\begin{split}
&\int_0^t(1+(t-s)^{(\alpha-1)p})\psi\left(w(s)\right)ds-\int_0^{t_0}(1+(t_0-s)^{(\alpha-1)p})\psi\left(w(s)\right)ds\\
&=\int_0^t \left(1+y^{(\alpha-1)p}\right)\psi\left(w(t-y)\right)dy-\int_0^{t_0}\left(1+y^{(\alpha-1)p}\right)\psi\left(w(t_0-y)\right)dy\\
&=\int_0^{t_0} \left(1+y^{(\alpha-1)p}\right)\left[\psi\left(w(t-y)\right)-\psi\left(w(t_0-y)\right)\right]dy+\int_{t_0}^t\left(1+y^{(\alpha-1)p}\right)\psi\left(w(t-y)\right)dy\\
&\geq 0,
\end{split}
\end{equation}
since $\psi$ is non-decreasing. We thus obtain
\begin{equation}\label{the3-16}
\begin{split} 
\sup\limits_{0\leq r\leq t}\mathbb{E}|x_{m+1}(r)-x_{k+1}(r)|^p&\leq C\int_{0}^t \left(1+(t-s)^{(\alpha-1)p}\right)\psi\left(w(s)\right) ds
\end{split}
\end{equation}
	
Taking limit as $m,k\rightarrow \infty$ and using the Fatou lemma, we obtain for every $\varepsilon>0$,
\begin{equation}\label{the3-17}
\begin{split} 
\lim\limits_{m,k\rightarrow\infty} w(t)\leq \varepsilon+C\int_{0}^t \left(1+(t-s)^{(\alpha-1)p}\right)\psi\left(\lim\limits_{m,k\rightarrow\infty}w(s)\right) ds.
\end{split}
\end{equation}
By the Bihari inequality yields
\begin{equation}\label{the3-18}
	\begin{split} 
\lim\limits_{m,k\rightarrow\infty}w(t)\leq W^{-1}\left(W(\varepsilon)+C\right),
	\end{split}
\end{equation}
where $W(\varepsilon)+C\in Dom(W^{-1})$, and $W(r)=\int_1^r\frac{1}{\psi(s)}ds$. 

By Assumption \ref{assumption1}, we get $\lim\limits_{\varepsilon\rightarrow 0}W(\varepsilon)=-\infty$ and $Dom(W^{-1})=(-\infty,W(\infty))$. Taking $\varepsilon\rightarrow 0$ gives
\begin{equation}\label{the3-19}
	\begin{split} 
\mathbb{E}\left(\sup\limits_{0\leq s\leq T}|x_m(s)-x_k(s)|^p\right)\rightarrow 0, \ m,k\rightarrow\infty,
	\end{split}
\end{equation}
indicating that ${x_k}$ is a Cauchy sequence in $L^p\left(\Omega;C\left([0,T];R^n\right)\right)$. We denote the limit by $x(t)$. Let $m\rightarrow \infty$ in (\ref{the3-19}), we have
\begin{equation}\label{the3-20}
	\begin{split} 
		\lim\limits_{k\rightarrow\infty}\mathbb{E}\left(\sup\limits_{0\leq s\leq T}|x(s)-x_k(s)|^p\right)\rightarrow 0.
	\end{split}
\end{equation}
Furthermore, as $k\rightarrow\infty$ in (\ref{the3-10}), we obtain
\begin{equation}\label{the3-21}
	\begin{split} 
		\sup\limits_{0\leq t\leq T}\mathbb{E}\left|x(t) \right|^p\leq C_{p,L,\alpha,x_0,T}.
	\end{split}
\end{equation}
	
\textbf{Uniqueness.} Let $x_1$ and $x_2$ be two solutions for (\ref{the1-1}) on the same probability space with $x_1(0)=x_2(0)$. By using a similar estimation method as above, we can get 
\begin{equation}\label{the3-22}
\begin{split} 
\mathbb{E}\left(\sup\limits_{0\leq s\leq t}|x_1(s)-x_2(s)|^p\right)\leq C\int_{0}^t \left(1+(t-s)^{(\alpha-1)p}\right)\psi\left(\mathbb{E}\left(\sup\limits_{0\leq r\leq s}\left|x_1(r)-x_2(r)\right|^p\right)\right) ds.
\end{split}
\end{equation}
The Bihari inequality implies that $x_1(t)=x_2(t)$, $t\in[0,T]$, $\mathbb{P}-a.s$.

The proof is completed.
\end{proof}

\section{Homogenization principle}
\label{section4}
\subsection{Main result}
In this section, we shall establish the homogenization principle $(\ref{the2-4})$ in the sense of mean square ($p=2$). First, fix the variable $x_\epsilon$ in $f\left(\frac{t}{\epsilon},x_\epsilon\right)$ and $g\left(\frac{t}{\epsilon},x_\epsilon\right)$ and then homogenize $t$:
\begin{equation}\label{the4-1}
\begin{split}
&\overline{f}(x_\epsilon(t))=\lim\limits_{T_1\rightarrow\infty}\frac{1}{T_1}\int_0^{T_1}f\left(\frac{s}{\epsilon},x_\epsilon(t)\right)ds,\\
&\overline{g}(x_\epsilon(t))=\lim\limits_{T_1\rightarrow\infty}\frac{1}{T_1}\int_0^{T_1}g\left(\frac{s}{\epsilon},x_\epsilon(t)\right)ds.\\
\end{split}
\end{equation}
Therefore, (\ref{the2-4}) can be transformed into
\begin{equation}\label{the4-2}
\begin{split}
x_\epsilon(t)&=x_0+\frac{1}{\Gamma(\alpha)}\int_0^t(t-s)^{\alpha-1}\overline{f}(x_\epsilon(s))ds+\frac{1}{\Gamma(\alpha)}\int_0^t(t-s)^{\alpha-1}\overline{g}(x_\epsilon(s))dB_s\\
&+\frac{1}{\Gamma(\alpha)}\int_0^t(t-s)^{\alpha-1}\left(f(s/\epsilon,x_\epsilon(s))-\overline{f}(x_\epsilon(s))\right)ds\\
&+\frac{1}{\Gamma(\alpha)}\int_0^t(t-s)^{\alpha-1}\left(g(s/\epsilon,x_\epsilon(s))-\overline{g}(x_\epsilon(s))\right)dB_s.
\end{split}
\end{equation}

In the following Theorem \ref{theorem4-4}, we will prove that the $p$th moment of the last two terms in (\ref{the4-2}) tend to $0$ as $\epsilon\rightarrow 0$. We define the homogenization equation for the approximation of $x_\epsilon$ by neglecting this remainder
\begin{equation}\label{the4-3}
\begin{split}
y_\epsilon(t)&=x_0+\frac{1}{\Gamma(\alpha)}\int_0^t(t-s)^{\alpha-1}\overline{f}(y_\epsilon(s))ds+\frac{1}{\Gamma(\alpha)}\int_0^t(t-s)^{\alpha-1}\overline{g}(y_\epsilon(s))dB_s.
\end{split}
\end{equation}

Suppose that $f\left(\frac{t}{\epsilon},x_\epsilon(t)\right)$ and $g\left(\frac{t}{\epsilon},x_\epsilon(t)\right)$ satisfy Assumption \ref{assumption1}. In addition, we need to assume that $\overline{f}(x_\epsilon(s))$ and $\overline{g}(x_\epsilon(s))$ satisfy the following homogenization conditions.

\begin{assumption}
	\label{assumption2}
For any $T_1\in(0,T]$ and $x\in\mathbb{R}^n$, there exists positive bounded functions $\gamma_1(t)$ and $\gamma_2(t)$ such that
\begin{equation}\label{the4-4}
	\begin{split}
&\left|\frac{1}{(T_1)^\alpha}\int_{0}^{T_1}\left(T_1-s\right)^{\alpha-1}\left(f(s,x)-\overline{f}(x)\right)  ds\right|^2\leq \gamma_1\left(T_1\right)\left(1+\left|x\right|^2\right),\\
&\frac{1}{(T_1)^{2\alpha-1}}\int_{0}^{T_1}(T_1-s)^{2(\alpha-1)}\left\Vert g(s,x)-\overline{g}(x)\right\Vert ^2  ds\leq \gamma_2\left(T_1\right)\left(1+\left|x\right|^2\right),
\end{split}
\end{equation}
where $\lim\limits_{T_1\rightarrow\infty}\gamma_1\left(T_1\right)=0$, \ $\lim\limits_{T_1\rightarrow\infty}\gamma_2\left(T_1\right)=0$.
\end{assumption}

\begin{remark}
Note that 
\begin{equation}
	\begin{split}
&\left|\frac{1}{(T_1)^\alpha}\int_{0}^{T_1}\left(T_1-s\right)^{\alpha-1}\left(f(s,x)-\overline{f}(x)\right)  ds\right|^2\leq \frac{1}{T_1}\int_0^{T_1} \left|f(s,x)-\overline{f}(x)\right|^2 ds,
	\end{split}
\end{equation}
which indicates that (\ref{the4-4}) is weaker than the widely used homogenization assumption:
\begin{equation}\label{the4-05}
	\begin{split}
\frac{1}{T_1}\int_0^{T_1} \left|f(s,x)-\overline{f}(x)\right|^2 ds\leq  \gamma_1\left(T_1\right)\left(1+\left|x\right|^2\right).
	\end{split}
\end{equation}
We would like to point out that (\ref{the4-05}) is not applicable to some long-term oscillatory systems (see an example in Section \ref{section4-2}).
\end{remark}

\begin{lemma}
	\label{lemma4-3}
The coefficients $\overline{f}(y_\epsilon(s))$ and $\overline{g}(y_\epsilon(s))$ in (\ref{the4-3}) satisfy Assumption \ref{assumption1}, thus its solution is unique. Furthermore, we have the following regularity estimate
\begin{equation}\label{the4-5}
	\begin{split}
\mathbb{E}\left(\sup\limits_{0\leq t\leq T}|y_\epsilon(t)|^2\right)\leq C_{L43}.
	\end{split}
\end{equation}
\end{lemma}

\begin{proof}
For $x,y\in \mathbb{R}^n$, by the H\"{o}lder inequality and (\ref{the4-1}), we have
\begin{equation}\label{the4-6}
\begin{split}
|\overline{f}(x)-\overline{f}(y)|^2&\leq \lim\limits_{T_1\rightarrow\infty} \frac{1}{T_1}\int_0^{T_1}\left|f\left(\frac{s}{\epsilon},x_\epsilon(t)\right)-f\left(\frac{s}{\epsilon},y_\epsilon(t)\right)\right|^2 ds\\
&\leq \lim\limits_{T_1\rightarrow\infty} \frac{1}{T_1}\int_0^{T_1}\psi(\left|x_\epsilon(t)-y_\epsilon(t)\right|^2) ds\\
&=\psi(\left|x_\epsilon(t)-y_\epsilon(t)\right|^2),
\end{split}
\end{equation}
and 
\begin{equation}\label{the4-7}
\begin{split}
\left\Vert\overline{g}(x)-\overline{g}(y)\right\Vert^2&\leq \psi(\left|x_\epsilon(t)-y_\epsilon(t)\right|^2).
\end{split}
\end{equation}
Moreover, by (\ref{the2-6})
\begin{equation}\label{the4-8}
	\begin{split}
|\overline{f}(0)|^2&\leq \lim\limits_{T_1\rightarrow\infty}\frac{1}{T_1}\int_0^{T_1} f\left(\frac{s}{\epsilon}, 0\right)ds\leq L.
	\end{split}
\end{equation}
Similarly, we can obtain $\Vert\overline{g}(0)\Vert^2$ is also bounded. Both $\overline{f}$ and $\overline{g}$ satisfy the conditions in Assumption \ref{assumption1}, so we can get $(\ref{the4-5})$.
\end{proof}

\begin{theorem}
\label{theorem4-4}
Suppose that Assumptions \ref{assumption1} and \ref{assumption2} hold. For $\alpha\in\left(\frac{1}{2},1\right)$, then
\begin{equation}\label{the4-9}
	\begin{split}
		\lim_{\epsilon\rightarrow 0}\mathbb{E}\left|x_{\epsilon}(t)-y_{\epsilon}(t)\right|^2=0,
	\end{split}
\end{equation}
for $t\in(0,T]$.
\end{theorem}

\begin{proof}
For any $t\in[0,T]$, we have
\begin{equation}\label{the4-10}
\begin{split}
&\mathbb{E}\left(|x_{\epsilon}(t)-y_{\epsilon}(t)|^2\right)\\
&\leq \frac{2}{\Gamma(\alpha)^2}\mathbb{E}\left(\left|\int_0^t(t-s)^{\alpha-1}\left(f(s/\epsilon,x_{\epsilon}(s))-\overline f(y_{\epsilon}(s))\right)ds\right|^2\right)\\
&+\frac{2}{\Gamma(\alpha)^2}\mathbb{E}\left(\left|\int_0^t(t-s)^{\alpha-1}\left(g(s/\epsilon,x_{\epsilon}(s))-\overline g(y_{\epsilon}(s))\right)dB_s\right|^2\right)\\
&=J_1+J_2.
\end{split}
\end{equation}
	
For $J_1$, we have
\begin{equation}\label{the4-11}
\begin{split}
J_1&\leq \frac{4}{\Gamma(\alpha)^2}\mathbb{E}\left(\left|\int_0^t(t-s)^{\alpha-1}\left(f(s/\epsilon,x_{\epsilon}(s))-f(s/\epsilon,y_{\epsilon}(s))\right)ds\right|^2\right)\\
&+\frac{4}{\Gamma(\alpha)^2}\mathbb{E}\left(\left|\int_0^t(t-s)^{\alpha-1}\left(f(s/\epsilon,y_{\epsilon}(s))-\overline f(y_{\epsilon}(s))\right)ds\right|^2\right)\\
&=:J_{11}+J_{12}.
\end{split}
\end{equation}
	
By the Cauchy-Schwarz inequality and Assumption \ref{assumption1}, we deduce that
\begin{equation}\label{the4-12}
\begin{split}
J_{11}&\leq\frac{4}{\Gamma(\alpha)^2}\left(\int_0^t(t-s)^{2(\alpha-1)}ds\right)\cdot\mathbb{E}\left(\int_0^{t}\Big\vert f(s/\epsilon,x_{\epsilon}(s))-f(s/\epsilon,y_{\epsilon}(s))\Big\vert^2ds\right)\\
&\leq Ct^{2\alpha-1}\int_0^{t}\psi\left(\mathbb{E}\left(|x_\epsilon(s)-y_\epsilon(s)|^2\right)\right)ds.
\end{split}
\end{equation}
With the help of the homogenization Assumption \ref{assumption2} and Lemma \ref{lemma4-3}, we obtain
\begin{equation}\label{the4-13}
\begin{split}
J_{12}&\leq\frac{4}{\Gamma(\alpha)^2}\mathbb{E}\left(\left|\epsilon\int_0^{\frac{t}{\epsilon}}(t-\epsilon w)^{\alpha-1}\left(f(w,y_{\epsilon}(\epsilon w))-\overline f(y_{\epsilon}(\epsilon w))\right)dw\right|^2\right)\\
&\leq \frac{4t^{2\alpha}}{\Gamma(\alpha)^2}\mathbb{E}\left(\left|\left(\frac{\epsilon}{t}\right)^\alpha\int_0^{\frac{t}{\epsilon}}\left(\frac{t}{\epsilon}- w\right)^{\alpha-1}\left(f(w,y_{\epsilon}(\epsilon w))-\overline f(y_{\epsilon}(\epsilon w))\right)dw\right|^2\right)\\
&\leq \frac{4t^{2\alpha}}{\Gamma(\alpha)^2}\gamma_1\left(\frac{t}{\epsilon}\right)\left(1+\mathbb{E}\left(\sup\limits_{0\leq w\leq \frac{t}{\epsilon}}|y_\epsilon(\epsilon w)|^2\right)\right)\\
&\leq C\gamma_1\left(\frac{t}{\epsilon}\right).
\end{split}
\end{equation}
	
For $J_2$, we have
\begin{equation}\label{the4-14}
\begin{split}
J_2&\leq \frac{4}{\Gamma(\alpha)^2}\mathbb{E}\left(\Big\vert\int_0^t(t-s)^{\alpha-1}\left(g(s/\epsilon,x_{\epsilon}(s))-g(s/\epsilon,y_{\epsilon}(s))\right)dB_s\Big\vert^2\right)\\
&+\frac{4}{\Gamma(\alpha)^2}\mathbb{E}\left(\Big\vert\int_0^t(t-s)^{\alpha-1}\left(g(s/\epsilon,y_{\epsilon}(s))-\overline g(y_{\epsilon}(s))\right)dB_s\Big\vert^2\right)\\
&=:J_{21}+J_{22}.
\end{split}
\end{equation}
	
By using the It\^{o} isometric formula and the Cauchy-Schwarz inequality, we obtain
\begin{equation}\label{the4-15}
\begin{split}
J_{21}&\leq \frac{4}{\Gamma(\alpha)^2}\mathbb{E}\left(\int_0^t(t-s)^{2\alpha-2}\left\Vert g(s/\epsilon,x_{\epsilon}(s))-g(s/\epsilon,y_{\epsilon}(s)\right\Vert^2ds\right)\\
&\leq C\int_0^t(t-s)^{2\alpha-2}\psi\left(\mathbb{E}\left(|x_\epsilon(s)-y_\epsilon(s)|^2\right)\right)ds.
\end{split}
\end{equation}
For the term $J_{22}$, by following the estimation way in (\ref{the4-13}), we have
\begin{equation}\label{the4-16}
\begin{split}
J_{22}&\leq \frac{4}{\Gamma(\alpha)^2}\mathbb{E}\left(\int_0^t(t-s)^{2\alpha-2}\left\Vert g(s/\epsilon,y_{\epsilon}(s))-\overline g(y_{\epsilon}(s))\right\Vert^2ds\right)\\
&\leq \frac{4t^{2\alpha-1}}{\Gamma(\alpha)^2}\mathbb{E}\left(\left(\frac{\epsilon}{t}\right)^{2\alpha-1}\int_0^{\frac{t}{\epsilon}}\left(\frac{t}{\epsilon}-w\right)^{2(\alpha-1)}\left\Vert g(w,y_\epsilon(\epsilon w))-\overline{g}(y_\epsilon(\epsilon w))\right\Vert^2dw\right)\\
&\leq C\gamma_2\left(\frac{t}{\epsilon}\right)\left(1+\mathbb{E}\left(\sup\limits_{0\leq w\leq \frac{t}{\epsilon}}|y_\epsilon(\epsilon w)|^2\right)\right)\\
&\leq C\gamma_2\left(\frac{t}{\epsilon}\right).
\end{split}
\end{equation}
	
Taking (\ref{the4-10})-(\ref{the4-16}) into consideration, we obtain
\begin{equation}\label{the4-19}
\begin{split}
\mathbb{E}\left(|x_{\epsilon}(t)-y_{\epsilon}(t)|^2\right)&\leq C\int_0^{t}\left[t^{2\alpha -1}+(t-s)^{2(\alpha-1)}\right]\psi\left(\mathbb{E}\left(|x_\epsilon(\tau)-y_\epsilon(\tau)|^2\right)\right)ds\\
&+C \left(\gamma_1\left(\frac{t}{\epsilon}\right)+\gamma_2\left(\frac{t}{\epsilon}\right)\right)
\end{split}
\end{equation}

Applying the Bihari inequality it is not difficult to obtain	
\begin{equation}\label{the4-20}
\begin{split}
&\mathbb{E}\left(\sup\limits_{0\leq t\leq T}|x_{\epsilon}(t)-y_{\epsilon}(t)|^p\right)\leq W^{-1}\left(W(q)+CT^{2\alpha}+CT^{2\alpha-1}\right),
\end{split}
\end{equation}
where
\begin{equation}\label{the4-21}
\begin{split}
&W(r)=\int_1^r\frac{1}{\psi(s)}ds, \quad q=C\left[\gamma_1\left(\frac{t}{\epsilon}\right)+ \gamma_2\left(\frac{t}{\epsilon}\right)\right].
\end{split}
\end{equation}
	
Let $\epsilon\rightarrow 0$, note that $\lim\limits_{\epsilon\rightarrow 0}\left(\gamma_1\left(\frac{t}{\epsilon}\right)+\gamma_2\left(\frac{t}{\epsilon}\right)\right)=0$. We then obtain $W(-\infty)\rightarrow 0$ and the desired result (\ref{the4-9}).
\end{proof}

\begin{remark}
\label{remark4-5}
From Theorem \ref{theorem4-4}, it can be seen that $\mathbb{E}\left(|x_\epsilon(t)-y_\epsilon(t)|^2\right)\lesssim \epsilon^{\mu} \ (\mu>0)$ as $\epsilon\rightarrow 0$. The convergence index $\mu$ is related to the function $\gamma_1$ and $\gamma_2$.
\end{remark}

\subsection{An illustrative example}
\label{section4-2}
We would like to point out that the homogenization assumption \ref{assumption2} proposed is applicable to some long-term oscillatory systems. Consider the following $f$ and $g$:
\begin{equation}\label{the4-22}
\begin{split}
f\left(\frac{t}{\epsilon},x_\epsilon\right)=2\cos^2\left(\frac{t}{\epsilon}\right)\sin(x_\epsilon(t)), \quad  g\left(\frac{t}{\epsilon},x_\epsilon\right)=\left(e^{-\frac{t}{\epsilon}}+1\right)x_\epsilon(t).
\end{split}
\end{equation}
We can see that the corresponding system (\ref{the1-1}) exhibits oscillatory behavior with time. By (\ref{the4-1}), for frozen slow component $x_\epsilon(t)$, we can compute
\begin{equation}\label{the4-23}
\begin{split}
&\overline{f}(x_\epsilon(t))=\lim\limits_{T_1\rightarrow\infty}\frac{1}{T_1}\int_0^{T_1}2\cos^2\left(\frac{s}{\epsilon}\right)\sin\left(x_\epsilon(t)\right)ds=\sin\left(x_\epsilon(t)\right),\\
&\overline{g}(x_\epsilon(t))=\lim\limits_{T_1\rightarrow\infty}\frac{1}{T_1}\int_0^{T_1}\left(e^{-\frac{s}{\epsilon}}+1\right)x_\epsilon(t)ds=x_\epsilon(t).
\end{split}
\end{equation}
We next verify that $\overline{f}(x_\epsilon(t))$ and $\overline{g}(x_\epsilon(t))$ satisfy the homogenization condition (\ref{the4-4}). For $\forall \delta\in(0,T_1)$, we derive

\begin{equation}
\begin{split}
&\left|\frac{1}{(T_1)^\alpha}\int_0^{T_1}(T_1-s)^{\alpha-1}\left(f(s,x_\epsilon)-\overline{f}(x_\epsilon)\right)ds\right|^2=\left|\frac{1}{(T_1)^\alpha}\int_0^{T_1}(T_1-s)^{\alpha-1}\cos(2s)\sin\left(x_\epsilon\right)dt \right|^2\\
&\leq \left(2\left|\frac{1}{(T_1)^\alpha}\int_0^{T_1-\delta}(T_1-s)^{\alpha-1}\cos(2s)dt \right|^2+2\left|\frac{1}{(T_1)^\alpha}\int_{T_1-\delta}^{T_1}(T_1-s)^{\alpha-1}dt \right|^2\right)\left(1+|x_\epsilon|^2\right)\\
&\leq \left(\frac{\delta^{2\alpha-2}}{(T_1)^{2\alpha}}+\frac{2\delta^{2\alpha}}{\alpha^2(T_1)^{2\alpha}}+\left|\frac{1-\alpha}{(T_1)^\alpha}\int_0^{T_1-\delta}\left(T_1-s\right)^{\alpha-2}\sin(2s)ds\right|^2\right)\left(1+|x_\epsilon|^2\right)\\
&\leq \left(\frac{2\delta^{2\alpha-2}}{(T_1)^{2\alpha}}+\frac{2\delta^{2\alpha}}{\alpha^2(T_1)^{2\alpha}}+\frac{1}{(T_1)^{2}}\right)\left(1+|x_\epsilon|^2\right)=:\gamma_1(T_1)\left(1+|x_\epsilon|^2\right).
\end{split}
\end{equation}
By using a similar method, we can also obtain
\begin{equation}
\begin{split}
&\frac{1}{(T_1)^{2\alpha-1}}\int_0^{T_1}(T_1-s)^{2\alpha-2}\left\Vert g(s,x_\epsilon)-\overline{g}(x_\epsilon)\right\Vert^2 ds\\
&\lesssim \left(\frac{1}{T_1}+\frac{\delta^{2\alpha-1}}{T^{2\alpha-1}}\right)(1+|x_\epsilon|^2)=:\gamma_2(T_1)(1+|x_\epsilon|^2).
\end{split}
\end{equation}
Since $\lim\limits_{T_1\rightarrow \infty}\gamma_1(T_1)=\lim\limits_{T_1\rightarrow \infty}\gamma_2(T_1)=0$, the simplified homogenization equation is
\begin{equation}
\begin{split}
y_\epsilon(t)&=x_\epsilon(0)+\frac{1}{\Gamma(\alpha)}\int_0^t (t-s)^{\alpha-1}\sin\left(y_\epsilon(s)\right)ds+\frac{1}{\Gamma(\alpha)}\int_0^t (t-s)^{\alpha-1}y_\epsilon(s)dB_s.\\
\end{split}
\end{equation}

\section{Numerical analysis}
\label{section5}
In this section, we derive an Euler-Maruyama numerical scheme for the non-autonomous system (\ref{the2-4}) and perform a convergence analysis. The numerical error of the corresponding autonomous system (\ref{the4-3}) is also obtained. This is beneficial for demonstrating the numerical advantages of homogenization in practical application problems. We consider convergence in the sense of mean square ($p=2$). Furthermore, we need to impose the Lipschitz condition $f$ and $g$.

\begin{assumption}
\label{assumption3}
There exists a positive constant $K$ such that $f$ and $g$ satisfies
\begin{equation}
	\begin{split}
\left|f(t_1,x_1)-f(t_2,x_2)\right|\vee \Vert g(t_1,x_1)-g(t_2,x_2)\Vert \leq K(|t_1-t_2|+|x_1-x_2|).
	\end{split}
\end{equation}
\end{assumption}

\subsection{Euler-Maruyama scheme}
We split the time interval $[0,T]$ into subintervals of equal size. Define $x_n=x(t_n)$, $t_n=n\Delta t$ \ $(0\leq n\leq N)$, where $\Delta t=T/N$ is the time step. The Euler-Maruyama scheme is used to discretize equation (\ref{the2-4}):
\begin{equation}\label{the5-1}
	\begin{split}
x_n&=x_0+\frac{1}{\Gamma(\alpha)}\sum\limits_{i=0}^{n-1}\int_{t_i}^{t_{i+1}}(t_n-s)^{\alpha-1}f\left(\frac{s}{\epsilon},x(s)\right)ds+\frac{1}{\Gamma(\alpha)}\sum\limits_{i=0}^{n-1}\int_{t_i}^{t_{i+1}}(t_n-s)^{\alpha-1}g\left(\frac{s}{\epsilon},x(s)\right)dB_s\\
&\approx x_0+\frac{1}{\Gamma(\alpha)}\sum\limits_{i=0}^{n-1}f\left(\frac{t_i}{\epsilon},x(t_i)\right)\int_{t_i}^{t_{i+1}}(t_n-s)^{\alpha-1}ds+\frac{1}{\Gamma(\alpha)}\sum\limits_{i=0}^{n-1}(t_n-t_i)^{\alpha-1}g\left(\frac{t_i}{\epsilon},x(t_i)\right)\int_{t_i}^{t_{i+1}}dB_s\\
&=x_0+\frac{1}{\Gamma(\alpha+1)}(\Delta t)^\alpha\sum\limits_{i=0}^{n-1}\left[(n-i)^\alpha-(n-i-1)^\alpha\right]f\left(\frac{t_i}{\epsilon},x(t_i)\right)\\
&+\frac{1}{\Gamma(\alpha)}(\Delta t)^{\alpha-1}\sum\limits_{i=0}^{n-1}(n-i)^{\alpha-1}g\left(\frac{t_i}{\epsilon},x(t_i)\right)\Delta B_i,
	\end{split}
\end{equation}
where $\Delta B_i=B(t_{i+1})-B(t_{i})\sim N(0,\Delta t)$ is a Gaussian distribution.

\subsection{Error analysis}
We first introduce an auxiliary continuous time stochastic process $z(t)$ on $[0,T]$ to help obtain the strong convergence of Scheme (\ref{the5-1}). Define the step function $\underline{s}$ such that $\underline{s}:=t_n$ for $s\in(t_n,t_{n+1}]$ and $0\leq n\leq N-1$,
\begin{equation}\label{the5-2}
\begin{split}
z(t)&=x_0+\frac{1}{\Gamma(\alpha)}\int_0^t(t-s)^{\alpha-1}f\left(\underline{s}/\epsilon,z(\underline{s})\right)ds+\frac{1}{\Gamma(\alpha)}\int_0^t(t-\underline{s})^{\alpha-1}g(\underline{s}/\epsilon,z(\underline{s}))dB_s.
\end{split}
\end{equation}

\begin{lemma}
\label{lemma5-1}
Let $x_n$ be the solution of the Euler-Maruyama scheme (\ref{the5-1}), $z(t)$ be the continuous time stochastic process defined by (\ref{the5-2}). For $0\leq n\leq N$, we have $z(t_n)=x_n$.
\end{lemma}
\begin{proof}
We use induction to prove this lemma. Firstly, $z(0)=x_0$ is obvious. Suppose that $z(t_m)=x_m$ holds for $0\leq m\leq n-1$. For $t=t_n$, we have
\begin{equation}\label{the5-3}
\begin{split}
z(t_n)&=x_0+\frac{1}{\Gamma(\alpha)}\int_0^{t_n}(t_n-s)^{\alpha-1}f\left(\underline{s}/\epsilon,z(\underline{s})\right)ds+\frac{1}{\Gamma(\alpha)}\int_0^{t_n}(t_n-\underline{s})^{\alpha-1}g\left(\underline{s}/\epsilon,z(\underline{s})\right)dB_s\\
&=x_0+\frac{1}{\Gamma(\alpha)}\sum\limits_{m=0}^{n-1}\int_{t_m}^{t_{m+1}}(t_n-s)^{\alpha-1}f\left(\underline{s}/\epsilon,z(\underline{s})\right)ds\\
&+\frac{1}{\Gamma(\alpha)}\sum\limits_{m=0}^{n-1}\int_{t_m}^{t_{m+1}}(t_n-\underline{s})^{\alpha-1}g\left(\underline{s}/\epsilon,z(\underline{s})\right)dB_s\\
&=x_0+\frac{1}{\Gamma(\alpha+1)}(\Delta t)^\alpha\sum\limits_{m=0}^{n-1}\left[(n-m)^\alpha-(n-m-1)^\alpha\right]f\left(\frac{t_m}{\epsilon},z(t_m)\right)\\
&+\frac{1}{\Gamma(\alpha)}(\Delta t)^{\alpha-1}\sum\limits_{m=0}^{n-1}(n-m)^{\alpha-1}g\left(\frac{t_m}{\epsilon},z(t_m)\right)\Delta B_m.
\end{split}
\end{equation}
We thus obtain $z(t_n)=x_n$.
\end{proof}

\begin{lemma}
	\label{lemma5-2}
For $z(t)$, the following moment estimate holds 
\begin{equation}\label{the5-4}
\begin{split}
\mathbb{E}\left(\sup\limits_{0\leq s\leq T}|z(s)|^2\right)\leq C_{L53},
\end{split}
\end{equation}
where $C_{L53}$ depends on $L, \alpha, x_0, ,K,T$.
\end{lemma}

By following arguments of Theorem \ref{theorem3-1}, we can easily obtain the desired result.\\

\begin{lemma}
	\label{lemma5-3}
For $t\in[t_n,t_{n+1})$ and $0\leq n\leq N-1$, we have the following estimate
\begin{equation}
	\label{the5-5}
	\begin{split}
		\mathbb{E}\left(|z(t)-z(t_n)|^2\right)\leq C_{L54}(\Delta t)^{2\alpha-1},
	\end{split}
\end{equation}
where $C_{L54}$ depends on $\alpha$ and $C_{L53}$.
\end{lemma}

\begin{proof}
For $t\in(t_n,t_{n+1})$, we have
\begin{equation}\label{the5-6}
\begin{split}
&z(t)-z(t_n)=\frac{1}{\Gamma(\alpha)}\int_{t_n}^{t}(t-s)^{\alpha-1}f(\underline{s}/\epsilon,z(\underline{s}))ds+\frac{1}{\Gamma(\alpha)}\int_{t_n}^{t}(t-\underline{s})^{\alpha-1}g(\underline{s}/\epsilon,z(\underline{s}))dB_s\\
&+\frac{1}{\Gamma(\alpha)}\int_{0}^{t_n}\left[(t-s)^{\alpha-1}-(t_n-s)^{\alpha-1}\right]f(\underline{s}/\epsilon,z(\underline{s}))ds\\
&+\frac{1}{\Gamma(\alpha)}\int_{0}^{t_n}\left[(t-\underline{s})^{\alpha-1}-(t_n-\underline{s})^{\alpha-1}\right]g(\underline{s}/\epsilon,z(\underline{s}))dB_s.
\end{split}
\end{equation}
This yields
\begin{equation}\label{the5-8}
\begin{split}
&\mathbb{E}\left(|z(t)-z(t_n)|^2\right)\leq \frac{2}{\Gamma(\alpha)^2}\mathbb{E}\left(\left|\int_{t_n}^{t}(t-s)^{\alpha-1}f(\underline{s}/\epsilon,z(\underline{s}))ds\right|^2\right)\\
&+\frac{2}{\Gamma(\alpha)^2}\mathbb{E}\left(\left|\int_{t_n}^{t}(t-\underline{s})^{2\alpha-2}g(\underline{s}/\epsilon,z(\underline{s}))^2ds\right|\right)\\
&+\frac{2}{\Gamma(\alpha)^2}\mathbb{E}\left(\left|\int_{0}^{t_n}\left[(t-s)^{\alpha-1}-(t_n-s)^{\alpha-1}\right]f(\underline{s}/\epsilon,y(\underline{s}))ds\right|^2\right)\\
&+\frac{2}{\Gamma(\alpha)^2}\mathbb{E}\left(\left|\int_{0}^{t_n}\left[(t-\underline{s})^{\alpha-1}-(t_n-\underline{s})^{\alpha-1}\right]^2g(\underline{s}/\epsilon,y(\underline{s}))^2ds\right|\right)\\
&=:\sum\limits_{i=1}^{4}N_i.
\end{split}
\end{equation}
	
By the Cauchy-Schwarz inequality, we can deduce that
\begin{equation}
\label{the5-9}
\begin{split}
N_1&\leq C\left(\int_{t_n}^t(t-s)^{2\alpha-2}ds\right)\int_{t_n}^{t}\mathbb{E}\left(|f(\underline{s}/\epsilon,z(\underline{s})|^2\right)ds\\
&\leq C(t-t_n)^{2\alpha},
\end{split}
\end{equation}
and $N_2\leq C(t-t_n)^{2\alpha-1}$.

Note that $(c^r-x^r)^2\leq (x^{2r}-c^{2r})$ for $0\leq x\leq c$ and $r\in(-1,0)$. For $N_3$ and $N_4$, we obtain
\begin{equation}
\label{the5-10}
\begin{split}
N_3+N_4&\leq C\int_{0}^{t_n}[(t-s)^{\alpha-1}-(t_n-s)^{\alpha-1}]^2\mathbb{E}|f(\underline{s}/\epsilon,x(s))|^2ds.\\
&+C\int_{0}^{t_n}[(t-\underline{s})^{\alpha-1}-(t_n-\underline{s})^{\alpha-1}]^2\mathbb{E}\Vert g(\underline{s}/\epsilon,x(s))\Vert^2ds\\
&\leq C\int_{0}^{t_n}\left[(t_n-s)^{2\alpha-2}-(t-s)^{2\alpha-2}\right]ds\\
&\leq C(t-t_n)^{2\alpha-1}.
\end{split}
\end{equation}
	
Combining (\ref{the5-8})-(\ref{the5-10}), we obtain
\begin{equation}\label{the5-11}
\begin{split}
\mathbb{E}\left(|z(t)-z(t_n)|^2\right)\leq C(t-t_n)^{2\alpha-1}+C(t-t_n)^{2\alpha}\leq C_{L54}(\Delta t)^{2\alpha-1}.
\end{split}
\end{equation}

The proof is completed.
\end{proof}

\begin{theorem}
	\label{theorem5-5}
Let $x_\epsilon(t)$ and $z(t)$ be the solutions to the FSDE (\ref{the2-4}) and (\ref{the5-2}), respectively. $x_n$ be the solution of the Euler-Maruyama scheme (\ref{the5-1}) corresponding to (\ref{the2-4}), we have the following error estimate
\begin{equation}\label{the5-12}
\begin{split}
\max\limits_{0\leq n\leq N}\mathbb{E}\left(|x_\epsilon(t_n)-x_n|^2\right)\lesssim(\Delta t)^{2\alpha-1}+\left(\frac{\Delta t}{\epsilon}\right)^2.
\end{split}
\end{equation}
\end{theorem}

\begin{proof}
For $t\in [t_n,t_{n+1})$ with $0\leq n\leq N-1$, we have
\begin{equation}\label{the5-13}
\begin{split}
&\mathbb{E}\left(|x_\epsilon(t)-z(t)|^2\right)\leq \frac{3}{\Gamma(\alpha)^2}\mathbb{E}\left(\left|\int_{0}^{t}(t-s)^{\alpha-1}\left(f(s/\epsilon,x_\epsilon(s))-f(\underline{s}/\epsilon,z(\underline{s}))\right)ds\right|^2\right)\\
&+\frac{3}{\Gamma(\alpha)^2}\mathbb{E}\left(\left|\int_{0}^{t}(t-s)^{\alpha-1}\left(g(s/\epsilon,x(s))-g(\underline{s}/\epsilon,z(\underline{s}))\right)ds\right|^2\right)\\
&+\frac{3}{\Gamma(\alpha)^2}\mathbb{E}\left(\left|\int_{0}^{t}\left((t-s)^{\alpha-1}-(t-\underline{s})^{\alpha-1}\right)g(\underline{s}/\epsilon,z(\underline{s}))ds\right|^2\right).
\end{split}
\end{equation}
	
By using Assumption \ref{assumption3}, Lemmas \ref{lemma5-2} and \ref{lemma5-3}, we can further obtain
\begin{equation}\label{the5-14}
\begin{split}
&\mathbb{E}\left(|x_\epsilon(t)-z(t)|^2\right)\leq C\sum\limits_{i=0}^{n-1}\int_{t_i}^{t_{i+1}}(t-s)^{2\alpha-2}\frac{(s-\underline{s})^2}{\epsilon^2}ds +C\int_{0}^{t}(t-s)^{2\alpha-2}\mathbb{E}\left(|x_\epsilon (s)-z(\underline{s})|^2\right)ds\\
&+C\int_{0}^t\left[(t-s)^{2\alpha-2}-(t-\underline{s})^{2\alpha-2}\right]ds\\
&\leq C\int_{0}^{t}(t-s)^{2\alpha-2}\mathbb{E}\left(|x_\epsilon (s)-z(s)|^2\right)ds+C\int_{0}^{t}(t-s)^{2\alpha-2}\mathbb{E}\left(|z(s)-z(\underline{s})|^2\right)ds\\
&+C\left(\frac{\Delta t}{\epsilon}\right)^2+C\int_{0}^t\left[(t-s)^{2\alpha-2}-(t-s+s-\underline{s})^{2\alpha-2}\right]ds\\
&\leq C\int_{0}^{t}(t-s)^{2\alpha-2}\mathbb{E}\left(|x_\epsilon (s)-z(s)|^2\right)ds+C(\Delta t)^{2\alpha-1}+C\left(\frac{\Delta t}{\epsilon}\right)^2\\
&+C\int_{0}^t\left[(t-s)^{2\alpha-2}-(t-s+\Delta t)^{2\alpha-2}\right]ds\\
&\leq C\int_{0}^{t}(t-s)^{2\alpha-2}\mathbb{E}\left(|x_\epsilon (s)-z(s)|^2\right)ds+C(\Delta t)^{2\alpha-1}+C\left(\frac{\Delta t}{\epsilon}\right)^2.
\end{split}
\end{equation}
	
By applying the Gronwall inequality, we obtain
\begin{equation}\label{the5-15}
\begin{split}
&\mathbb{E}\left(|x_\epsilon(t)-z(t)|^2\right)\leq C\left((\Delta t)^{2\alpha-1}+\left(\frac{\Delta t}{\epsilon}\right)^2\right).
\end{split}
\end{equation}
By Lemma \ref{lemma5-1}, (\ref{the5-12}) can be obtained immediately.
\end{proof}

\begin{remark}
It can be seen from (\ref{the5-12}) that the numerical error is not only related to the time step $\Delta t$ but also to the scale parameter $\epsilon$. For the homogenized system (\ref{the4-3}), we similarly use the Euler-Maruyama scheme for discretization to obtain the numerical solution $y_n$. Based on Theorem \ref{remark4-5}, we obtain the error 
\begin{equation}
\max\limits_{0\leq n\leq N}\mathbb{E}\left(|x_\epsilon(t_n)-y_n|^2\right)\lesssim (\Delta t)^{2\alpha-1}+\epsilon^\mu.
\end{equation}
\end{remark}

\begin{remark}
\label{remark5-7}
From the results of theoretical analysis, for multiscale non-autonomous systems with small scale paramete $\epsilon$, we need to use a small time step for calculation to avoid the error introduced by $\epsilon$. Balanced discretization errors are given for $(\Delta t)^{2\alpha-1}\approx \left(\frac{\Delta t}{\epsilon}\right)^2$, i.e., for $\Delta t\approx \epsilon^{\frac{2}{3-2\alpha}}$. Obtaining accurate numerical results in this way undoubtedly increases the computational cost. For the homogenized autonomous system, it can be seen that as $\epsilon$ decreases, the error between the numerical solution and the exact solution also decreases. Therefore, using the same time step to calculate the homogenized system has higher numerical accuracy, which is also demonstrated by our numerical results.
\end{remark}

\section{Numerical experiments}
\label{section6}
In this section, the first numerical example is given to test the convergence of scheme (\ref{the5-1}), and the second numerical example is shown to verify the theoretical result of Theorem \ref{theorem4-4}.  In the following simulations, the expectation is approximated by sample average \cite{cao2015numerical}. We define the mean square errors as
\begin{equation}\label{the6-1}
	\begin{split}
&e_{\Delta t}=\max\limits_{1\leq m\leq n}\left(\frac{1}{2000}\sum\limits_{i=1}^{2000}\Vert x^{(n,\epsilon)}(t_m,\omega_i)-x^{(2n,\epsilon)}(t_m,\omega_i)\Vert^2\right)^{1/2},\\
&e_{\epsilon}=\max\limits_{1\leq m\leq n}\left(\frac{1}{2000}\sum\limits_{i=1}^{2000}\Vert x^{(n,\epsilon)}(t_m,\omega_i)-x^{(n,2\epsilon)}(t_m,\omega_i)\Vert^2\right)^{1/2},
	\end{split}
\end{equation}
where $\omega_i$ denotes the $i$th sample path. All of computations are performed by using a MATLAB (R2023b) subroutine on a laptop computer with the Intel(R) Core(TM) i7-8550U CPU @ 1.80GHz and 8.0G RAM.

\subsection{Convergence test}
Consider the following non-autonomous FSDEs

\begin{equation}\label{the6-2}
	\begin{split}
^CD_{0}^{\alpha}x_\epsilon(t)=\frac{t}{\epsilon}x_\epsilon(t)dt+x_\epsilon(t)dB_t, \quad t\in[0,T],
	\end{split}
\end{equation}
with initial value $x_\epsilon(0)=0.1$. 

We select two different $\alpha$ ($\alpha=0.9$ and $\alpha=0.7$) to test the convergence rates of $\Delta t$ and $\epsilon$. We first set $T=0.1$ and fix $\epsilon=1$ to calculate the error with different $\Delta t$. As we can observe from Table \ref{table1}, our numerical method $(\alpha-0.5)$-order convergence for $\Delta t$. In order to test the most critical result, which is the convergence rate of $\epsilon$, we set $T=10^{-6}$ and fix $\Delta t=10^{-8}$ and calculate the error with different $\epsilon$. Table \ref{table2} shows the proper scaling in $\epsilon^{-1}$. These numerical
results are consistent with the error analysis in Theorem \ref{theorem5-5}.

\begin{table}[H]
	\caption{Convergence rates of $\Delta t$ for scheme (\ref{the5-1}) with $\alpha=0.9$ and $\alpha=0.7$.}\label{table1}
	\centering
	\scalebox{0.65}{
\resizebox{\linewidth}{!}{
	\begin{tabular}{c c c c c} \hline  		
		\scriptsize   $\Delta t$   & \scriptsize  $e_{\Delta t}$ ($\alpha=0.9$) & \scriptsize   Order & \scriptsize  $e_{\Delta t}$ ($\alpha=0.7$) & \scriptsize  Order \\ \hline
		\scriptsize  1/80	& \scriptsize  4.19E-3 &  & \scriptsize   \scriptsize  2.10E-2 &  \\
		\scriptsize  1/160	& \scriptsize  3.06E-3  & \scriptsize  0.45 & \scriptsize 1.85E-2  & \scriptsize  0.18 \\  
		\scriptsize  1/320	& \scriptsize  2.40E-3    &  \scriptsize  0.35 & \scriptsize  1.64E-2 & \scriptsize  0.17 \\  
		\scriptsize  1/640	& \scriptsize  1.79E-3    & \scriptsize  0.42 &  \scriptsize  1.42E-2   & \scriptsize  0.21 \\  \hline
		\end{tabular}}
	}
\end{table}

\begin{table}[H]
	\caption{Convergence rates of $\epsilon$ for scheme (\ref{the5-1}) with $\alpha=0.9$ and $\alpha=0.7$.}\label{table2}
	\centering
	\scalebox{0.65}{
		\resizebox{\linewidth}{!}{
			\begin{tabular}{c c c c c} \hline  		
				\scriptsize   $\epsilon$   & \scriptsize  $e_{\epsilon}$ ($\alpha=0.9$) & \scriptsize   Order & \scriptsize  $e_{\epsilon}$ ($\alpha=0.7$) & \scriptsize  Order \\ \hline
				\scriptsize  4E-8		& \scriptsize  2.70E-6 &  & \scriptsize   \scriptsize  5.07E-5 &  \\
				\scriptsize  8E-8	& \scriptsize  1.35E-6  & \scriptsize  -1.00 & \scriptsize 2.54E-5  & \scriptsize  -1.00 \\  
				\scriptsize  1.6E-7	& \scriptsize  6.74E-7     &  \scriptsize  -1.00 & \scriptsize  1.27E-5 & \scriptsize  -1.00 \\  
				\scriptsize  3.2E-7	& \scriptsize  3.37E-7    & \scriptsize  -1.00 &  \scriptsize  6.34E-6   & \scriptsize  -1.00 \\  \hline
		\end{tabular}}
	}
\end{table}

\subsection{Stochastic fractional diffusion equation}
For the stochastic fractional diffusion equation \cite{niu2020inverse}
\begin{equation}\label{the6-3}
	\begin{split}
		^CD_{0}^{\alpha}x_\epsilon(t,\textbf{x})=\Delta x_\epsilon(t,\textbf{x})+f\left(\frac{t}{\epsilon},x_\epsilon(t,\textbf{x})\right)dt+g\left(\frac{t}{\epsilon},x_\epsilon(t,\textbf{x})\right)dB_t, \quad (t,\textbf{x})\in[0,T]\times \Omega,
	\end{split}
\end{equation}
with initial value $x_\epsilon(0,\textbf{x})$ and the homogeneous Dirichlet boundary condition, it can be simplified to the following FSDEs system (see \cite{wang2019wellposedness})
\begin{equation}\label{the6-4}
	\begin{split}
^CD_{0}^{\alpha}x_\epsilon(t)+\lambda x_\epsilon(t)=f\left(\frac{t}{\epsilon}, x_\epsilon(t)\right)dt+g\left(\frac{t}{\epsilon}, x_\epsilon(t)\right)dB_t, \quad t\in[0,T],
	\end{split}
\end{equation}
with initial value $x_\epsilon(0)=\frac{1}{2}$ and $\lambda>0$. We set the data
\begin{equation}
\begin{split}
f=2\cos^2\left(\frac{t}{\epsilon}\right)\sin\left(x_\epsilon(t)\right),\quad g\left(\frac{t}{\epsilon},x_\epsilon\right)=\left(e^{-\frac{t}{\epsilon}}+1\right)x_\epsilon(t), \quad \lambda=\frac{1}{2}.
\end{split}
\end{equation}

It is easy to verify that $f$ and $g$ satisfy the Assumption \ref{assumption1} and Assumption \ref{assumption3}. The system has a unique solution $x_\epsilon$ given by
\begin{equation}\label{the6-6}
\begin{split}
x_\epsilon(t)&=x_\epsilon(0)+\frac{1}{\Gamma(\alpha)}\int_0^t (t-s)^{\alpha-1}\left(2\cos^2\left(\frac{s}{\epsilon}\right)\sin\left(x_\epsilon(s)\right)-\lambda x_\epsilon(s)\right)ds\\
&+\frac{1}{\Gamma(\alpha)}\int_0^t (t-s)^{\alpha-1}\left(e^{-\frac{s}{\epsilon}}+1\right)x_\epsilon(s)dB_s.
\end{split}
\end{equation}

In Section \ref{section4-2}, we have analyzed the homogenized coefficients $\overline{f}(x_\epsilon(t))$ and $\overline{g}(x_\epsilon(t))$, and obtained the homogenization equation

\begin{equation}\label{the6-9}
	\begin{split}
		y_\epsilon(t)&=x_\epsilon(0)+\frac{1}{\Gamma(\alpha)}\int_0^t (t-s)^{\alpha-1}\left(\sin\left(y_\epsilon(s)\right)-\lambda y_\epsilon(s)\right)ds+\frac{1}{\Gamma(\alpha)}\int_0^t (t-s)^{\alpha-1}y_\epsilon(s)dB_s.\\
	\end{split}
\end{equation}

In the simulation, we set $\alpha=0.9$ and choose a smaller time step $\Delta t=1/1024$ to compute (\ref{the6-6}) as the reference solution. For different $\epsilon$, we use a coarse time step $\Delta t=1/256$ to compute systems (\ref{the6-6}) and (\ref{the6-9}), and define the errors based on the reference solution as $Ex_{\epsilon}$ and $Ey_{\epsilon}$, respectively. It can be seen from Figure \ref{figure6-1}(a) that when the scale difference of the variables is not significant, the result obtained by computing the original system (\ref{the6-6}) with the same step size is obviously more accurate than that obtained by computing the homogenized system (\ref{the6-9}). However, as the scale increases, the result is reversed. It is observed from Figure \ref{figure6-1}(b) that the computation of homogenization system (\ref{the6-9}) achieves better accuracy, which confirms the result of Theorem \ref{theorem4-4}. 

Compared to computing the original system, computing the homogenized system can use larger step sizes while maintaining accuracy, thus reducing computational costs and improving efficiency in computing application problems. Moreover, the larger the scale difference of variables, the better the effect of using the homogenization method.

\begin{figure}[H]
\centering
\subfigure[$\epsilon=10^{-1}$]{
\includegraphics[scale=0.55]{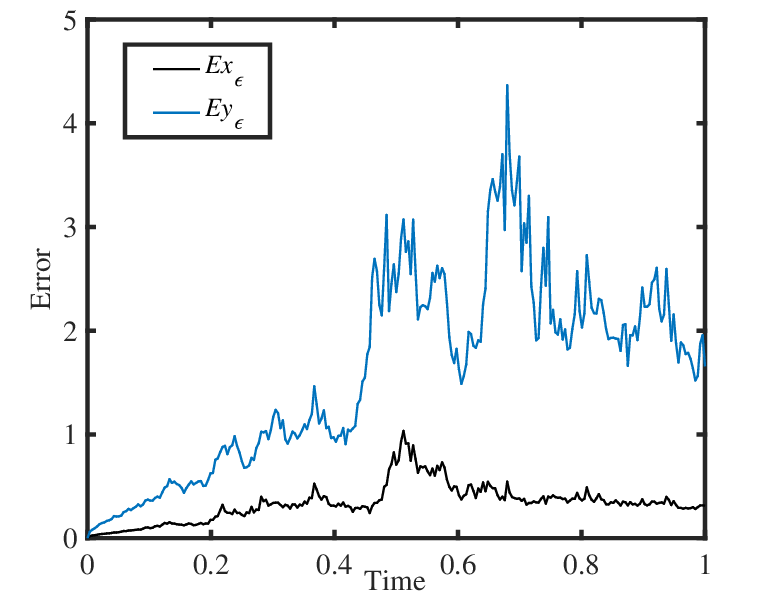}
  }
\subfigure[$\epsilon=10^{-4}$]{
\includegraphics[scale=0.55]{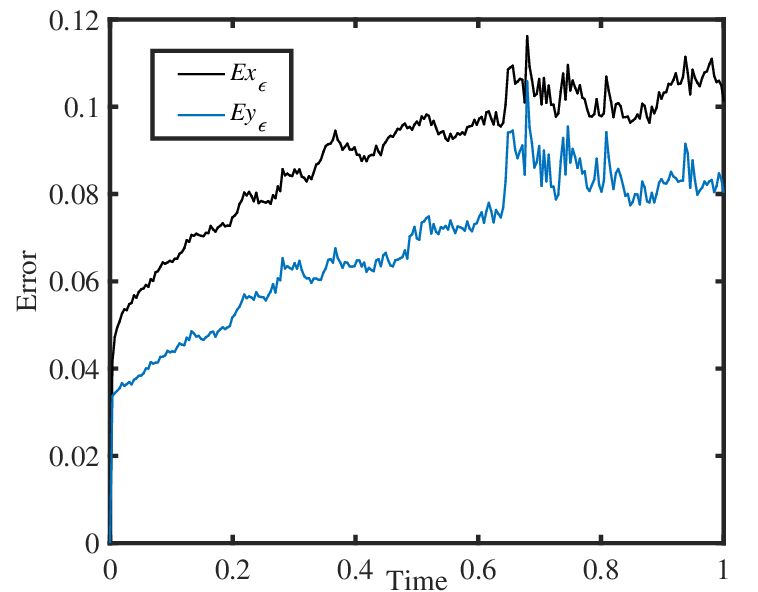}}
\caption{The error of the original variable $x_\epsilon$ and the homogenized variable $y_\epsilon$ with respect to the reference solution at different scales.}
\label{figure6-1}
\end{figure}

\section{Conclusions and remarks}
\label{section7}
In this paper, we establish the well-posedness and homogenization theory of FSDEs under the non-Lipschitz condition. To explore the advantages of the homogenization method in computational application problems, we construct an Euler-Maruyama scheme and carry out a rigorous numerical analysis. The numerical experiments verified our theoretical results. For FSDEs with significant differences in the scales of variables, computing the homogenized autonomous system avoids the errors introduced by the discretization of $t/\epsilon$, which makes the numerical computation more efficient.

For future work, we shall extend the analytical framework proposed in this paper to stochastic partial differential equations. In addition, the efficient computation of temporal multiscale and multiphysics field coupling problems with fractional derivatives (e.g., the plaque growth problem \cite{wang2023fast}) using the homogenization method is also a topic of our future research.

\section*{Acknowledgements}
This work is partially supported by the National Natural Science Foundation of China 12371388,
11861131004 and 11771040.

\bibliographystyle{siam}
\bibliography{Ref}

\begin{thebibliography}{10}

\bibitem{ainsworth2020fractional}
{\sc M.~Ainsworth and Z.~Mao}, {\em Fractional phase-field crystal modelling:
  analysis, approximation and pattern formation}, IMA Journal of Applied
  Mathematics, 85 (2020), pp.~231--262.

\bibitem{blanc2023homogenization}
{\sc X.~Blanc and C.~Le~Bris}, {\em Homogenization Theory for Multiscale
  Problems}, vol.~21, Springer, 2023.

\bibitem{cao2015numerical}
{\sc W.~Cao, Z.~Zhang, and G.~E. Karniadakis}, {\em Numerical methods for
  stochastic delay differential equations via the wong--zakai approximation},
  SIAM Journal on Scientific Computing, 37 (2015), pp.~A295--A318.

\bibitem{doan2020euler}
{\sc T.~S. Doan, P.~T. Huong, P.~E. Kloeden, and A.~M. Vu}, {\em
  Euler--maruyama scheme for caputo stochastic fractional differential
  equations}, Journal of Computational and Applied Mathematics, 380 (2020),
  p.~112989.

\bibitem{guo2023averaging}
{\sc Z.~Guo, X.~Han, and J.~Hu}, {\em Averaging principle for stochastic caputo
  fractional differential equations with non-lipschitz condition}, Fractional
  Calculus and Applied Analysis,  (2023), pp.~1--18.

\bibitem{Khsminskii1968}
{\sc R.~Khsminskii}, {\em On the principle of averaging the it\^{o} stochastic
  differential equations}, Kibernetika, 4 (1968), pp.~260--279.

\bibitem{li2017fractional}
{\sc L.~Li, J.-G. Liu, and J.~Lu}, {\em Fractional stochastic differential
  equations satisfying fluctuation-dissipation theorem}, Journal of Statistical
  Physics, 169 (2017), pp.~316--339.

\bibitem{li2023existence}
{\sc M.~Li and J.~Wang}, {\em The existence and averaging principle for caputo
  fractional stochastic delay differential systems}, Fractional Calculus and
  Applied Analysis, 26 (2023), pp.~893--912.

\bibitem{mao2007stochastic}
{\sc X.~Mao}, {\em Stochastic differential equations and applications},
  Elsevier, 2007.

\bibitem{metzler2000random}
{\sc R.~Metzler and J.~Klafter}, {\em The random walk's guide to anomalous
  diffusion: a fractional dynamics approach}, Physics reports, 339 (2000),
  pp.~1--77.

\bibitem{niu2020inverse}
{\sc P.~Niu, T.~Helin, and Z.~Zhang}, {\em An inverse random source problem in
  a stochastic fractional diffusion equation}, Inverse Problems, 36 (2020),
  p.~045002.

\bibitem{ouaddah2021fractional}
{\sc A.~Ouaddah, J.~Henderson, J.~Nieto, and A.~Ouahab}, {\em A fractional
  bihari inequality and some applications to fractional differential equations
  and stochastic equations}, Mediterranean Journal of Mathematics, 18 (2021),
  p.~242.

\bibitem{pavliotis2008multiscale}
{\sc G.~A. Pavliotis and A.~Stuart}, {\em Multiscale methods: averaging and
  homogenization}, vol.~53, Springer Science \& Business Media, 2008.

\bibitem{podlubny1998fractional}
{\sc I.~Podlubny}, {\em Fractional differential equations: an introduction to
  fractional derivatives, fractional differential equations, to methods of
  their solution and some of their applications}, Elsevier, 1998.

\bibitem{Podlubny1999}
{\sc I.~Podlubny}, {\em {Fractional Differential Equations ({N}ew {Y}ork:
  Academic)}},  (1999).

\bibitem{sakthivel2013existence}
{\sc R.~Sakthivel, P.~Revathi, and Y.~Ren}, {\em Existence of solutions for
  nonlinear fractional stochastic differential equations}, Nonlinear Analysis:
  Theory, Methods \& Applications, 81 (2013), pp.~70--86.

\bibitem{shen2022averaging}
{\sc G.~Shen, J.~Xiang, and J.-L. Wu}, {\em Averaging principle for
  distribution dependent stochastic differential equations driven by fractional
  brownian motion and standard brownian motion}, Journal of Differential
  Equations, 321 (2022), pp.~381--414.

\bibitem{wang2019wellposedness}
{\sc H.~Wang and X.~Zheng}, {\em Wellposedness and regularity of the
  variable-order time-fractional diffusion equations}, Journal of Mathematical
  Analysis and Applications, 475 (2019), pp.~1778--1802.

\bibitem{wang2020note}
{\sc W.~Wang, S.~Cheng, Z.~Guo, and X.~Yan}, {\em A note on the continuity for
  caputo fractional stochastic differential equations}, Chaos: An
  Interdisciplinary Journal of Nonlinear Science, 30 (2020).

\bibitem{wang2024error}
{\sc Z.~Wang and P.~Lin}, {\em Error analysis of a highly efficient and
  accurate temporal multiscale method for a fractional differential system},
  Chaos, Solitons \& Fractals, 179 (2024), p.~114447.

\bibitem{wang2023fast}
{\sc Z.~Wang, P.~Lin, and L.~Zhang}, {\em A fast front-tracking approach and
  its analysis for a temporal multiscale flow problem with a fractional order
  boundary growth}, SIAM Journal on Scientific Computing, 45 (2023),
  pp.~B646--B672.

\bibitem{weinan2011principles}
{\sc E.~Weinan}, {\em Principles of multiscale modeling}, Cambridge University
  Press, 2011.

\bibitem{xu2019averaging}
{\sc W.~Xu, W.~Xu, and S.~Zhang}, {\em The averaging principle for stochastic
  differential equations with caputo fractional derivative}, Applied
  Mathematics Letters, 93 (2019), pp.~79--84.

\bibitem{yang2022numerical}
{\sc Z.~Yang}, {\em Numerical approximation and error analysis for
  caputo--hadamard fractional stochastic differential equations}, Zeitschrift
  f{\"u}r angewandte Mathematik und Physik, 73 (2022), p.~253.

\end{thebibliography}
\end{document}